\documentclass[12pt]{article}

\usepackage{amsmath,amssymb,amsthm}
\usepackage{geometry}
\usepackage{enumitem}
\usepackage{titlesec}

\newtheorem{theorem}{Theorem}[section]
\newtheorem{definition}[theorem]{Definition} 
\newtheorem*{remark}{Remark}

\date{}

\begin{document}
\title{Quantitative Soft-to-Hard Terminal Constraint Convergence for the Heat Equation}
\author{
Sung-Sik Kwon\\
Department of Mathematics and Physics\\
North Carolina Central University\\
Durham, NC 27707, USA\\
\texttt{skwon@nccu.edu}
}

\maketitle

\begin{abstract}
We investigate an optimal control problem for the one-dimensional heat
equation with a prescribed terminal state. To approximate the hard
terminal constraint, we consider quadratic terminal penalization and
analyze the convergence of the optimal controls and terminal states
as the penalty parameter \(\alpha\to\infty\).

We develop a unified spectral treatment of terminal reachability,
minimum-energy control, and quadratic-penalty convergence. Using modal
decompositions, we characterize finite-energy admissible terminal
targets and construct the corresponding minimum-energy hard-constrained
control. For admissible targets, we establish strong convergence of
penalized optimal controls and terminal states. Under suitable spectral
summability conditions, we obtain regularity-dependent convergence
estimates of order \(O(\alpha^{-(1/2+\theta)})\) for the terminal-state
error, \(0\le\theta\le1/2\), and \(O(\alpha^{-\theta})\) for the
optimal-control error, \(0<\theta\le1\).

The analysis connects these results explicitly with classical
Tikhonov regularization in the heat-equation setting. A
finite-dimensional prototype illustrates the underlying projection
structure, while companion numerical investigations provide
supporting evidence for the predicted convergence behavior.
\end{abstract}


\vspace{1em}


\section{Introduction}

Optimal control problems governed by partial differential equations (PDEs) arise
naturally in many areas of science and engineering, including thermal
regulation, diffusion processes, and distributed parameter systems. In such
problems, one seeks to influence the behavior of a distributed system through a
control input while minimizing a prescribed performance criterion. A
comprehensive theoretical framework for PDE-constrained optimization has been
developed in classical works such as \cite{Hinze2009,Lions1971,Troltzsch2010}.

A common objective in many applications is to steer the system toward a
prescribed terminal state. Such requirements lead to \emph{hard terminal
constraints}, in which the terminal state is enforced exactly. However,
enforcing hard terminal conditions directly often presents analytical and
computational challenges. For this reason, \emph{penalized terminal constraint}
formulations are widely employed, where the terminal condition is incorporated
into the objective functional through a quadratic penalty term. Penalty,
regularization, and relaxation methods play an important role in constrained
optimal control and PDE-constrained optimization because they avoid the
difficulties associated with enforcing exact constraints while retaining strong
approximation properties \cite{Hinze2009,Troltzsch2010}.

The treatment of terminal-state requirements is also closely connected
with the classical theory of exact and approximate controllability for
distributed parameter systems. Foundational optimization-based approaches
were developed by Glowinski and Lions
\cite{GlowinskiLions1994,GlowinskiLions1995}, while Carthel, Glowinski,
and Lions \cite{CarthelGlowinskiLions1994} investigated controllability
of the heat equation from a numerical perspective. Related penalization
approaches for parabolic equations were studied by Lions
\cite{Lions1997}. 
These works provide the broader controllability background for the
present study.

Alternative approaches to handling terminal constraints via penalization have
been studied in the optimal control literature for both ODEs and PDEs;
see, for example, \cite{GugatZuazua2016,Maurer2008}. In contrast to
quadratic penalties, certain nonsmooth penalty formulations possess an
exactness property: under suitable assumptions, their minimizers
coincide with those of the constrained problem for all sufficiently
large penalty parameters. Gugat and Zuazua \cite{GugatZuazua2016}
established such results for linear evolution equations using an
\(L^1\)-type penalty, while Maurer \cite{Maurer2008} investigated
related exact-penalty results in the ODE setting.

Of particular relevance to the present formulation, Cao \cite{Cao2001}
investigated quadratic terminal-penalty approximations for parabolic
control problems, including convergence and numerical approximation.
Joshi and Kumar \cite{JoshiKumar2005} established strong convergence of
penalized optimal pairs for linear parabolic equations. Hou, Imanuvilov,
and Kwon \cite{HouImanuvilovKwon2006} developed explicit eigenseries
representations for terminal-state tracking and exact controllability
problems governed by linear parabolic equations.

The central objective of this paper is to analyze the relationship between
penalized and hard terminal constraints for the one-dimensional heat equation
as the penalty parameter \(\alpha\to\infty\). Under suitable modal summability
assumptions on the terminal mismatch, we derive explicit algebraic convergence
rates for both the terminal-state error and the control error. In particular,
the terminal-state convergence rate ranges from \(O(\alpha^{-1/2})\) under the
baseline admissibility condition to the sharper \(O(\alpha^{-1})\) rate under
stronger modal summability conditions, while the control convergence rate is
of order \(O(\alpha^{-\theta})\) for \(0<\theta\le1\).

Unlike nonsmooth exact-penalty approaches, which can
recover the constrained solution beyond a finite
penalty threshold, quadratic terminal penalization
generally yields asymptotic convergence as
\(\alpha\to\infty\). The latter formulation is closely
related to classical Tikhonov regularization, where
spectral source conditions yield regularity-dependent
convergence rates
\cite{EnglHankeNeubauer1996,Natterer1984}.
The present work develops this connection explicitly
for terminal control of the one-dimensional heat
equation, expressing admissibility and convergence
conditions in terms of the heat-equation spectrum
and the prescribed terminal correction.

To motivate the analysis, we first examine a simple finite-dimensional prototype
inspired by a simplified rocket motion model. Although the control belongs to an
infinite-dimensional space, the terminal condition reduces to a single scalar
moment constraint. This structure reveals a projection geometry in Hilbert
space, consistent with classical results in functional optimization
\cite{Luenberger1969}. The prototype serves as an instructive model for the
infinite-dimensional case considered later.

The analysis is then extended to a linear parabolic control problem governed by
the one-dimensional heat equation. Using modal decompositions, we characterize
the class of admissible terminal states that can be achieved under hard terminal
constraints. This characterization reflects the smoothing properties of the
heat equation and provides a natural framework for analyzing penalized terminal
formulations.

The present work focuses primarily on theoretical
analysis. Additional numerical investigations are
reported in the companion study by Kwon and Nunda
\cite{KwonNunda2026}, including independent
Crank--Nicolson finite-difference validation,
numerical illustrations of the admissibility
threshold, and regularity-dependent intermediate
terminal-error convergence rates. These computations
provide supporting evidence for the theoretical
framework developed here.

\paragraph{Contributions.}
The principal contribution of this work is a unified spectral treatment
of terminal reachability, minimum-energy control, and quantitative
quadratic-penalty convergence for the one-dimensional heat equation.
Within this framework, the main results and distinguishing features
are summarized as follows:
\begin{itemize}
  \item We provide an explicit spectral characterization of admissible
  terminal states, together with the corresponding minimum control
  energy.
  \item We derive modal representations of the minimum-energy
  hard-constrained control and the quadratically penalized controls,
  establishing their strong soft-to-hard convergence.
  \item We express regularity-dependent convergence estimates in terms
  of explicit spectral conditions for the heat equation, obtaining
  terminal-state rates from \(O(\alpha^{-1/2})\) to
  \(O(\alpha^{-1})\) and control rates of
  \(O(\alpha^{-\theta})\), \(0<\theta\le1\), under the
  corresponding modal summability assumptions.
  \item We use a simple rocket-motion prototype to illustrate the
  projection structure underlying the terminal-control problem and
  motivate its infinite-dimensional extension.
\end{itemize}

\section{Finite-Dimensional Prototype}

Before studying the heat equation, we first examine a simpler control problem
that illustrates the same soft-to-hard terminal constraint mechanism in
a more elementary setting. This example, referred to here as the \emph{rocket
problem}, serves as a useful prototype that clarifies the structure of terminal
penalization before addressing the infinite-dimensional heat equation.

Although the control is a function of time and therefore belongs to the
infinite-dimensional space \(L^2(0,T)\), the terminal condition depends only on a
single scalar moment of the control. Consequently, the minimum-energy control
lies in a one-dimensional subspace and admits an explicit formula.

\subsection{Hard Terminal Constraint}

Let \(T>0\). We consider the following second-order system, which may be
interpreted as a simplified model of vertical rocket motion:
\begin{equation} \label{rocket_system}
\begin{cases}
    y''(t)=v(t)-1, & 0<t<T, \\[4pt]
    y(0)=0, \\[2pt]
    y'(0)=0.
\end{cases}
\end{equation}
Here, \(y(t)\) represents position, \(v(t)\) represents thrust, and the constant
term \(1\) corresponds to normalized gravitational acceleration. Throughout this
section, we assume
\[ 
    v\in L^2(0,T).
\]
We impose the hard terminal condition
\[
    y(T)=y_T, 
\]
which models the requirement that the system attain a prescribed terminal
state, such as a desired final altitude.

In this paper, we use the term \emph{hard terminal constraint} to refer to the
exact equality constraint \(y(T)=y_T\), in contrast to the penalized
(\emph{soft}) formulation introduced later.

The hard terminal control problem is
\[
    \min_{v\in L^2(0,T)} \frac12\|v\|_{L^2(0,T)}^2
    \quad \text{subject to} \quad \eqref{rocket_system}
    \text{ and } y(T)=y_T.
\]
Integrating twice gives
\[
    y(t) = \int_0^t (t-s)(v(s)-1)\,ds.
\]
At \(t=T\),
\[
  \begin{aligned}
	y(T) & = \int_0^T (T-s)(v(s)-1)\,ds \\ 
	& = \int_0^T (T-s)v(s)\,ds - \frac{T^2}{2}.
  \end{aligned}
\]
Thus, the terminal condition becomes
\[
  \int_0^T (T-s)v(s)\,ds = y_T+\frac{T^2}{2}.
\]
Define
\[
    d = y_T+\frac{T^2}{2}, \qquad g(t)=T-t.
\]
Then
\[
  \langle v,g\rangle=d,
\]
where \(\langle \cdot,\cdot \rangle\) denotes the \(L^2(0,T)\) inner product.

Therefore, the hard terminal problem reduces to
\[
  \min_{v\in L^2(0,T)}
  \frac12\|v\|_{L^2(0,T)}^2 \quad \text{subject to} \quad
  \langle v,g\rangle=d.
\]
The solution is obtained as the orthogonal projection of the origin onto the
affine hyperplane
\[
  \{v:\langle v,g\rangle=d\},
\]
a classical result in Hilbert space optimization \cite{Luenberger1969}. Hence
\[
	v^*(t) = \frac{d}{\|g\|^2}g(t).
\]
Since
\[
  \|g\|^2 = \frac{T^3}{3},
\]
we obtain
\[
    v^*(t) = \frac{3}{T^3} \left( y_T+\frac{T^2}{2} \right) (T-t).
\]
This is the unique minimum-energy control satisfying the hard terminal
constraint.

Substituting \(v^*\) into the state representation yields
\[
    y^*(t) = \frac{1}{2T^3} \left( y_T+\frac{T^2}{2} \right) t^2(3T-t) -
\frac{t^2}{2}.
\]
This explicit polynomial expression completes the characterization of the
optimal pair \((v^*,y^*)\). One readily verifies that
\[
    y^*(T)=y_T.
\]

\subsection{Penalized Terminal Constraint Formulation}

Here we replace the hard terminal condition with a quadratic terminal penalty.

For \(\alpha>0\), define
\[
    J_\alpha(v) = \frac12\|v\|^2 + \frac{\alpha}{2} |\langle v,g\rangle-d|^2.
\]
Since this functional is strictly convex and coercive on \(L^2(0,T)\), there
exists a unique minimizer \(v_\alpha^*\).
The optimality condition is
\[
    v_\alpha^* = -\alpha(\langle v_\alpha^*, g\rangle - d)g,
\]
which shows that \(v_\alpha^*\) lies in \(\mathrm{span}\{g\}\). Writing
\[
    v_\alpha^*(t) = c_\alpha g(t),
\]
we obtain
\[
    c_\alpha = \frac{\alpha d}{1+\alpha a}, \qquad a=\|g\|^2.
\]
Hence
\[
    v_\alpha^*(t) = \frac{\alpha}{1+\alpha a} d\,g(t) = \frac{\alpha a}{1+\alpha a} v^*(t).
\]
Substituting \(v_\alpha^*\) into the state representation yields
\[
    y_\alpha^*(t) = \beta_\alpha \left( y^*(t)+\frac{t^2}{2} \right) - \frac{t^2}{2},
    \qquad \beta_\alpha=\frac{\alpha a}{1+\alpha a}.
\]
From these expressions, it is already clear that
\[
    v_\alpha^* \to v^*,\qquad y_\alpha^* \to y^*
\]
as \(\alpha\to\infty\). We now quantify these convergence rates explicitly.

\subsection{Soft-to-Hard Convergence Rates}

Since
\[
    v_\alpha^* = \beta_\alpha v^*, \qquad
    \beta_\alpha=\frac{\alpha a}{1+\alpha a},
\]
we have
\[
    v_\alpha^*-v^* = (\beta_\alpha-1)v^* = -\frac{1}{1+\alpha a} v^*.
\]
Thus
\[
    \|v_\alpha^*-v^*\| = \frac{\|v^*\|}{1+\alpha a} =  O(\alpha^{-1}).
\]

Similarly,
\[
    y_\alpha^*(t)-y^*(t) = (\beta_\alpha-1)
    \left( y^*(t)+\frac{t^2}{2} \right).
\]
Since \(y^*(t)+t^2/2\) is bounded on \([0,T]\),
we obtain the uniform convergence estimate
\[
    \|y_\alpha^*-y^*\|_{C([0,T])}
    = O(\alpha^{-1}).
\]

Evaluating the state difference formula at \(t=T\), we obtain
\[
    y_\alpha^*(T)-y_T = -\frac{d}{1+\alpha a},
\]
and therefore
\[
    |y_\alpha^*(T)-y_T| =  O(\alpha^{-1}).
\]

Thus, the penalized optimal pair \((v_\alpha^*,y_\alpha^*)\) converges to the
hard-constrained solution \((v^*,y^*)\), with the control error, state error,
and terminal mismatch all decaying at rate \( O(\alpha^{-1}) \) as
\(\alpha\to\infty\).

\subsection{Interpretation as a Prototype}

Although the rocket problem is formulated in the
infinite-dimensional space \(L^2(0,T)\), its terminal
constraint depends only on the scalar moment
\[
\langle v,g\rangle=d.
\]
Consequently, the optimal control lies in
\(\operatorname{span}\{g\}\), reducing the problem
to a single scalar coefficient.

The heat equation has an analogous modal structure,
but its terminal condition involves infinitely many
Fourier coefficients,
\[
d_n=y_{T,n}-e^{-\lambda_nT}y_{0,n},
\]
where \(y_{0,n}\) and \(y_{T,n}\) are the modal
coefficients of the initial and terminal states.
Thus, the heat equation may be viewed as an
infinite-mode analogue of the rocket problem.

A fundamental difference is that not every terminal
state is reachable by a finite-energy control in
the parabolic setting. This leads naturally to
the admissible target set \(\mathcal A_T(y_0)\),
which characterizes terminal states satisfying
the hard constraint.

The rocket prototype therefore provides a
conceptual bridge to the infinite-dimensional
analysis developed in the following section.

\section{Parabolic Terminal Control Problem}

We extend the soft-to-hard terminal constraint framework developed for the
finite-dimensional rocket prototype to an infinite-dimensional parabolic control
problem. As a canonical model, we consider the one-dimensional heat equation
with distributed control.

Let \(\Omega=(0,1)\) denote the spatial domain, and let \(T>0\) be a fixed
final time. We define the space-time domain
\[
    Q := (0,1)\times(0,T).
\]

\subsection{State Equation}

We consider the one-dimensional heat equation with a distributed control:
\begin{equation} \label{state_equation}
  \begin{cases}
    y_t - y_{xx} = u, & (x,t)\in Q, \\[4pt]
    y(0,t)=0, \quad y(1,t)=0, & t\in(0,T), \\[4pt]
    y(x,0)=y_0(x),            & x\in(0,1).
  \end{cases}
\end{equation}
Here, \(y(x,t)\) denotes the state variable, \(u(x,t)\) denotes the distributed
control, and \(y_0(x)\) is the prescribed initial state. For clarity, we
distinguish the distributed control \(u(x,t)\) used in the present parabolic
problem from the scalar control \(v(t)\) introduced in the effectively one-dimensional 
rocket prototype of Section~2.
 
Throughout this work, we assume that
\[
    y_0 \in L^2(0,1),\qquad u \in L^2(Q).
\]
Under these assumptions, the state equation admits a unique weak solution
satisfying
\[
y \in
L^2(0,T;H_0^1(0,1))
\cap
H^1(0,T;H^{-1}(0,1)).
\]
In particular,
\[
y \in C([0,T];L^2(0,1)),
\]
so the terminal value \(y(\cdot,T)\) is well-defined.
This is a standard well-posedness result for parabolic equations
(see, for example, \cite{Evans2010,Troltzsch2010}).

Our objective is to steer the system toward a prescribed terminal profile
\(y_T(x)\) at time \(T\), using either exact terminal constraints or penalized
terminal constraints, as described in the following subsections.

\subsection{Hard Terminal Constraint Formulation}

We first consider the formulation in which the terminal profile is enforced
exactly. This corresponds to the ideal situation in which the system is required
to reach a prescribed terminal state at the final time.

Let \(y_T(x)\) denote a desired terminal profile. We measure control effort using
the quadratic energy functional
\[
    J(u) = \frac12\int_0^T\int_0^1 |u(x,t)|^2\,dx\,dt.
\]
The terminal condition is imposed exactly:
\[
    y(x,T)=y_T(x).
\]
Thus, the hard terminal control problem is
\[
\begin{aligned}
    \min_{u \in L^2(Q)} &\quad J(u) \\
    \text{subject to} &\quad \eqref{state_equation}
    \quad \text{and} \quad y(x,T)=y_T(x).
\end{aligned}
\]
This formulation seeks a control of minimum energy that steers the system
exactly to the desired terminal state.

However, unlike the rocket prototype of Section~2, not every
terminal profile \(y_T(x)\) is reachable using finite-energy controls. This
limitation reflects the smoothing property of the heat equation and leads
naturally to the introduction of an admissible set of terminal targets, which is
defined in the following subsection.

\subsection{Admissible Terminal Targets}
\label{sec:admissible-targets}
Here we analyze the structure of the hard terminal constraint problem for
admissible targets.

To understand why restrictions on terminal profiles arise, it is useful to
examine the behavior of individual Fourier modes. Solutions of the heat equation
exhibit strong decay in high-frequency components: modes associated with larger
eigenvalues decay more rapidly in time. As a result, reconstructing fine-scale
oscillations at the terminal time requires increasingly large control effort.

This suggests that admissible terminal states should satisfy a
frequency-dependent restriction reflecting the smoothing action of the heat
equation. In particular, high-frequency components are increasingly difficult to
realize using finite-energy controls.

We make this observation precise through a modal decomposition of the
system, which leads to an explicit characterization of the admissible set of
terminal targets. The characterization developed below is closely related to
classical reachable-set and minimum-energy control theory for parabolic
equations; see, for example, \cite{ChenRosier2022, GlowinskiLions1994,Troltzsch2010}. 
The present formulation is included here to provide a concrete framework for
the subsequent soft-to-hard penalized convergence analysis.

\paragraph{Admissible Terminal Targets}

\begin{definition}[Admissible Terminal Target Set]
Let
  \[
    e_n(x)=\sqrt{2}\sin(n\pi x), \qquad \lambda_n=(n\pi)^2,
  \]
denote the eigenfunctions and eigenvalues of the negative Laplacian
\(-\frac{d^2}{dx^2}\) on \((0,1)\) with homogeneous Dirichlet boundary conditions.
Expand
  \[
    y_0 =\sum_{n=1}^{\infty} y_{0,n} e_n,\qquad y_T 
    = \sum_{n=1}^{\infty} y_{T,n} e_n.
  \]
We define the admissible target set
  \[
    \mathcal A_T(y_0)
    =
    \left\{
    y_T\in L^2(0,1):
    \sum_{n=1}^{\infty}
    \frac{2\lambda_n}{1-e^{-2\lambda_n T}}
    \left|y_{T,n} - e^{-\lambda_n T}y_{0,n} \right|^2 <\infty
    \right\}.
  \]

We shall show that \(\mathcal A_T(y_0)\) characterizes
the terminal profiles that are reachable from \(y_0\)
using finite-energy controls.
\end{definition}

\paragraph{Existence and Uniqueness}

\begin{theorem}[Existence and Uniqueness --
Hard Terminal Constraint]
Let \(y_0,y_T\in L^2(0,1)\). Then there exists
a control \(u\in L^2(Q)\) satisfying
\[
y(\cdot,T)=y_T
\]
if and only if
\[
y_T\in\mathcal A_T(y_0).
\]
When this condition holds, there exists
a unique minimum-energy control
\(u^*\in L^2(Q)\) satisfying
\(y(\cdot,T)=y_T\), which minimizes
\[
J(u)=\frac12\|u\|_{L^2(Q)}^2.
\]
\end{theorem}

\begin{proof}[Proof sketch]
The proof is based on expanding the solution of
the heat equation into eigenfunctions of the
Laplacian. Writing
\[
y(x,t)=\sum_{n=1}^{\infty}y_n(t)e_n(x),
\qquad
u(x,t)=\sum_{n=1}^{\infty}u_n(t)e_n(x),
\]
the state equation decouples into independent
scalar ordinary differential equations:
\[
y_n'(t)+\lambda_n y_n(t)=u_n(t).
\]

Necessity follows from the modal terminal
constraints and the Cauchy--Schwarz inequality,
which imply that every finite-energy control
reaching \(y_T\) satisfies the admissibility
condition.

Conversely, when \(y_T\in\mathcal A_T(y_0)\),
each mode corresponds to a scalar minimum-energy
control problem analogous to the rocket prototype
studied in Section~2. Solving these independent
problems yields explicit formulas for the optimal
modal controls. The admissibility condition
ensures that their combined energy is finite.

Summing over all modes produces the unique
minimum-energy control, whose explicit
representation is given below.
A complete proof is provided in Appendix~A.
\end{proof}

\paragraph{Explicit Modal Representation}

The modal structure underlying the admissibility condition leads to an explicit
representation of the optimal control.

Define
\[
    d_n = y_{T,n} - e^{-\lambda_n T}y_{0,n}.
\]
Then the optimal control admits the representation
\[
    u^*(x,t) = \sum_{n=1}^{\infty} \frac{2\lambda_n
    e^{-\lambda_n(T-t)}}{1-e^{-2\lambda_n T}} d_n e_n(x).
\]

This representation shows that the hard terminal control problem decouples into
infinitely many independent modal problems. Each mode behaves analogously to the
scalar constraint identified in the rocket prototype of Section~2.

\subsection{Penalized Terminal Control Formulation}

While the hard terminal constraint formulation provides an ideal mathematical
description, it suffers from a fundamental limitation: not every terminal
profile can be achieved using finite-energy controls. This restriction motivates
the introduction of a penalized formulation in which terminal mismatch is
incorporated into the objective functional.

Penalized terminal constraints are widely used in optimization and numerical practice, since the resulting optimization problems remain solvable even when the desired terminal state is not exactly reachable \cite{Bertsekas1982,NocedalWright2006}.

\paragraph{Penalized Formulation}
For \(\alpha>0\), define the penalized functional
\[
	J_\alpha(u) = \frac12 \int_0^T \int_0^1 |u(x,t)|^2 \,dx\,dt 
	+ \frac{\alpha}{2} \int_0^1 |y(x,T)-y_T(x)|^2 \,dx.
\]
The penalized terminal control problem is
\[
  \min_{u \in L^2(Q)} \; J_\alpha(u)
  \quad \text{subject to} \quad  \eqref{state_equation},
\]
where the terminal constraint is replaced by a quadratic penalty term.

Unlike the hard constraint case, this formulation admits a unique solution for
every terminal target
\[
	y_T \in L^2(0,1).
\]

\paragraph{Existence and Uniqueness}

\begin{theorem}[Existence and Uniqueness -- Penalized Terminal Constraint]
Let
  \[
	y_0 \in L^2(0,1), \qquad y_T \in L^2(0,1), \qquad \alpha>0.
  \]
Then there exists a unique optimal control
  \[
	u_\alpha^* \in L^2(Q)
  \]
minimizing \(J_\alpha(u)\) subject to the state equation \eqref{state_equation}.
\end{theorem}

\begin{proof}[Proof sketch]
The terminal map \(u\mapsto y(\cdot,T)\) is
affine and continuous from \(L^2(Q)\) into
\(L^2(0,1)\). Consequently, \(J_\alpha\) is
strictly convex and weakly lower semicontinuous
on \(L^2(Q)\). Moreover,
\[
J_\alpha(u)\ge
\frac12\|u\|_{L^2(Q)}^2,
\]
so \(J_\alpha\) is coercive.
Existence follows from the direct method in the
calculus of variations, and strict convexity
ensures uniqueness; see, for example,
\cite{Troltzsch2010}.
\end{proof}

\begin{remark}
The penalized formulation replaces the hard terminal constraint
with a quadratic penalty that measures terminal mismatch.
As \(\alpha\) increases, the optimization places increasing
emphasis on matching the terminal state. When
\(y_T\in\mathcal A_T(y_0)\), the penalized optimal controls
converge to the unique minimum-energy control of the
hard terminal problem, as established in the following section.
\end{remark}

\subsection{Soft-to-Hard Convergence}\label{sec:soft-to-hard-convergence}

We now establish that penalized controls converge to the hard terminal control
as the penalty parameter increases.

\begin{theorem}[Soft-to-Hard Convergence]
Assume
  \[
	y_0\in L^2(0,1), \qquad y_T\in \mathcal A_T(y_0).
  \]
Let \(u^*\) denote the optimal control for the hard terminal constraint, and
\(u_\alpha^*\) the optimal control for the penalized formulation.

Then
  \[
	u_\alpha^* \to u^* \quad \text{in }L^2(Q) \quad \text{as} \quad \alpha\to\infty.
  \]

Moreover,
  \[
	y_\alpha(\cdot,T) \to y_T \quad \text{in }L^2(0,1).
  \]
\end{theorem}

\begin{proof}[Proof sketch]
The proof relies on modal decomposition.
The penalized problem reduces to independent
scalar minimization problems, whose explicit
solutions converge to the corresponding
hard-constrained modal controls as
\(\alpha\to\infty\).

The admissibility condition provides a
summable bound on the modal control errors.
The dominated convergence theorem for series
then yields strong convergence in \(L^2(Q)\).
Convergence of the terminal states follows
from the continuity of the terminal control map.

A detailed proof is provided in Appendix~B.
\end{proof}

\subsection{Rates of Soft-to-Hard Convergence}
\label{sec:rates-of-convergence}
The convergence result above can be refined to obtain a quantitative rate.
Define
\[
    d_n = y_{T,n}-e^{-\lambda_nT}y_{0,n}, \qquad
    a_n = \frac{1-e^{-2\lambda_nT}}{2\lambda_n}.
\]
\begin{theorem}[Rates of Convergence]

\begin{enumerate}
\item Let \(0\le \theta\le \frac12\). Assume
\[
\sum_{n=1}^{\infty} \frac{|d_n|^2}{a_n^{1+2\theta}} < \infty.
\]
Then there exists a constant \(C_\theta>0\), independent of \(\alpha\), 
such that
\[
    \|y_\alpha(\cdot,T)-y_T\|_{L^2(0,1)}
    \le \frac{C_\theta}{\alpha^{\frac12+\theta}}.
\]
\item Let \(0<\theta\le 1\). Assume
\[
    \sum_{n=1}^{\infty} \frac{|d_n|^2}{a_n^{1+2\theta}} < \infty.
\]
Then there exists a constant \(D_\theta>0\), independent of \(\alpha\), 
such that
\[
    \|u_\alpha^*-u^*\|_{L^2(Q)} \le \frac{D_\theta}{\alpha^\theta}.
\]
\end{enumerate}
\end{theorem}

\begin{remark}
The estimates above quantify how the penalized formulation approaches the hard
terminal control problem as the penalty parameter \(\alpha\to\infty\). The
terminal-state error admits the baseline convergence rate
\[
O(\alpha^{-1/2})
\]
under the admissibility condition, while stronger modal summability assumptions
yield the sharper rate
\[
O(\alpha^{-1}).
\]
The control convergence estimates exhibit a similar regularity--rate tradeoff:
stronger decay of the modal discrepancies \(d_n\) yields faster convergence of
the penalized controls.

These assumptions reflect the smoothing property of the heat equation. Highly
oscillatory terminal components require increasingly large control energy to
reproduce, so stronger convergence rates are obtained only when the
high-frequency modal coefficients decay sufficiently rapidly.

In particular, finite sine expansions satisfy all summability conditions
appearing in the theorem automatically, since only finitely many modal
coefficients are nonzero. Consequently, finite-dimensional spectral
approximations attain the endpoint convergence rates
\[ 
    O(\alpha^{-1})
\]
for both the terminal-state and control errors.

A detailed proof of these estimates, based on modal analysis and energy
arguments, is provided in Appendix~C.
\end{remark}

\begin{remark}[Connection with Tikhonov regularization]
Let
\[
K:L^2(Q)\longrightarrow L^2(0,1)
\]
denote the terminal control-to-state operator, defined by
\[
Ku=\int_0^T e^{(T-s)\Delta_D}u(s)\,ds,
\]
where \(\Delta_D\) is the Dirichlet Laplacian.
Writing
\[
d=y_T-e^{T\Delta_D}y_0,
\]
the penalized control problem is equivalent, after
division of the objective by \(\alpha\), to
\[
\min_{u\in L^2(Q)}
\left\{
\frac12\|Ku-d\|_{L^2(0,1)}^2
+\frac{\lambda}{2}\|u\|_{L^2(Q)}^2
\right\},
\qquad \lambda=\alpha^{-1}.
\]

For the heat equation, the eigenvalues of \(KK^*\)
are precisely the modal coefficients
\[
a_n=\frac{1-e^{-2\lambda_nT}}{2\lambda_n}.
\]
For an admissible terminal target and
\(\theta\ge0\), the summability condition
\[
\sum_{n=1}^{\infty}
\frac{|d_n|^2}{a_n^{1+2\theta}}<\infty
\]
is equivalent to the spectral source condition
\[
u^*\in
\operatorname{Range}\bigl((K^*K)^\theta\bigr),
\]
where \(u^*\) denotes the minimum-energy
hard-constrained control.

Indeed, the coefficients of \(u^*\) in the right singular-vector basis
of \(K\) are \(d_n/\sqrt{a_n}\), and the corresponding nonzero
eigenvalues of \(K^*K\) are \(a_n\). Consequently, the abstract source
condition is equivalent to
\[
\sum_{n=1}^{\infty}
\frac{|d_n|^2}{a_n^{1+2\theta}}<\infty,
\]
which is precisely the weighted spectral summability condition used
in our convergence estimates.

Thus, the convergence-rate exponents obtained
above agree with classical Tikhonov estimates
under spectral source conditions; see, for example,
\cite{EnglHankeNeubauer1996,Natterer1984}.
The present analysis provides an explicit
heat-equation realization of this framework,
connecting the source conditions to the
admissibility of terminal targets, the
minimum-energy control, and the spectral
regularity of the prescribed terminal correction.
\end{remark}

\section{Numerical Illustration}

Although the primary focus of this paper is theoretical,
we briefly summarize numerical results illustrating the
soft-to-hard convergence behavior and the influence of
spectral regularity. The original computations include
analytical modal evaluations and finite-difference
approximations. More extensive numerical investigations
are reported in the companion study by Kwon and Nunda
\cite{KwonNunda2026}.

\subsection{Rocket Prototype}

We first consider the rocket prototype introduced in
Section~2. Using the explicit analytical solutions,
the terminal-state, control, and state-trajectory errors
were evaluated for penalty parameters ranging from
\(\alpha=1\) to \(10^6\).

The observed asymptotic slopes were close to \(-1\)
for all three error quantities, consistent with the
theoretical first-order convergence rate
\(O(\alpha^{-1})\).

An independent time-discrete optimization with
\(N=80\) time steps yielded an observed slope of
approximately \(-0.9997\), providing additional
numerical confirmation of the predicted behavior.

\subsection{Heat Equation: First-Order Convergence}

For the heat equation, we consider the single-mode
benchmark
\[
y_0(x)=0,\qquad
y_T(x)=\sin(\pi x),\qquad T=1.
\]

The terminal-state error was evaluated using both
the exact modal representation and a finite-difference
approximation for penalty parameters ranging from
\(\alpha=1\) to \(10^4\).

The observed terminal-error slopes were approximately
\(-0.9883\) for the modal computation and
\(-0.9849\) for the finite-difference approximation.

Because the terminal target contains only one Fourier
mode, it satisfies the spectral summability conditions
for the endpoint rate. Consequently, the theoretical
estimate
\[
\|y_\alpha(\cdot,T)-y_T\|_{L^2(0,1)}
=O(\alpha^{-1})
\]
applies, in agreement with the observed slopes.

Further independent Crank--Nicolson validation,
including approximately second-order convergence
under simultaneous space--time mesh refinement,
is presented in \cite{KwonNunda2026}.

\subsection{Admissibility and Spectral Regularity}

The companion study by Kwon and Nunda
\cite{KwonNunda2026} investigates terminal
corrections with Fourier coefficients
\[
d_n\propto n^{-p},
\]
where \(p\) determines the spectral decay of the
prescribed terminal correction.

For this family, the admissibility condition becomes
\[
p>\frac32.
\]
The numerical experiments illustrate this threshold:
truncated hard-control norms tend to stabilize as
the number of modes increases for \(p>3/2\), while
they exhibit continued growth for \(p\le3/2\)
over the tested truncation levels.

The same study examines regularity-dependent
terminal-error convergence rates using exact modal
computations. For \(p=1.75\) and \(p=2.00\),
the observed log--log slopes were approximately
\(-0.62559\) and \(-0.74999\), respectively,
compared with the asymptotic predictions
\(-0.625\) and \(-0.750\).

These predictions follow from the asymptotic
analysis of the exact modal error formula
presented in \cite{KwonNunda2026} and are
consistent with the general regularity-dependent
convergence estimates of Theorem~3.5.

For the smoother target corresponding to \(p=3.00\),
the observed slope was approximately \(-0.99970\),
consistent with first-order convergence.

These results illustrate the influence of spectral
regularity on terminal-error convergence. In
particular, intermediate rates arise for targets
with slower spectral decay, while sufficiently
rapid decay yields first-order behavior.

\subsection{Discussion}

The numerical results illustrate the principal
features of the theoretical soft-to-hard convergence
analysis. The rocket prototype and single-mode
heat-equation benchmark exhibit approximately
first-order convergence, consistent with the
corresponding theoretical estimates.

The additional investigations reported in
\cite{KwonNunda2026} provide numerical evidence
for the influence of terminal-target admissibility
and spectral regularity on control energy and
terminal-error convergence rates.

Together, these results support the theoretical
framework developed in this paper. The computations
are illustrative rather than proofs of convergence;
the detailed numerical methodology and validation
results are presented in the companion study.

\section{Conclusion}
We presented a unified spectral treatment of terminal reachability,
minimum-energy control, and quadratic terminal penalization for the
one-dimensional heat equation. For admissible terminal targets, we
established strong convergence of the penalized optimal controls to
the minimum-energy hard-constrained control as \(\alpha\to\infty\),
together with convergence of the corresponding terminal states.
The analysis makes explicit the connection with classical Tikhonov
regularization in this heat-equation setting.

Under suitable spectral summability conditions,
we obtained quantitative convergence estimates
for both the terminal-state and control errors.
The terminal-state estimates range from
\(O(\alpha^{-1/2})\) to \(O(\alpha^{-1})\),
depending on the spectral regularity of the
terminal correction, 
while the control-error estimates are of order
\(O(\alpha^{-\theta})\), \(0<\theta\le1\),
under suitable spectral summability conditions.

A finite-dimensional prototype illustrated the
underlying projection structure, and numerical
illustrations were consistent with the predicted
convergence behavior. Additional investigations
reported in the companion study
\cite{KwonNunda2026} provide supporting evidence
for the influence of admissibility and spectral
regularity on control energy and terminal-error
convergence rates.

The analysis is restricted to the one-dimensional
heat equation with distributed controls and
quadratic terminal penalization. Possible extensions
include more general parabolic equations and
control configurations, as well as further
investigation of discretization effects.

%%%%%%%%%%%%%%%%%%%%%% APPENDIX % %%%%%%%%%%%%%%%%%%%%%%%%%%

\appendix
\renewcommand{\thesection}{Appendix \Alph{section}}
\section{Proof of the Hard Terminal Control Theorem}

We provide a detailed proof of the existence
and uniqueness of the minimum-energy hard-constrained
control, together with the characterization of
admissible terminal targets, as discussed in
Section~\ref{sec:admissible-targets}.

\begin{proof}
We begin by expanding the state and control variables in terms of the
eigenfunctions of the negative Laplacian. Standard well-posedness theory
for the heat equation gives
\[
y\in L^2(0,T;H_0^1(0,1))
\cap C([0,T];L^2(0,1)),
\qquad u\in L^2(Q).
\]
By Fubini's theorem, \(y(\cdot,t),u(\cdot,t)\in L^2(0,1)\) for almost
every \(t\in(0,T)\).
Because \(\{e_n\}_{n=1}^\infty\) forms an orthonormal basis of
\(L^2(0,1)\), we have, for almost every \(t\in(0,T)\), the expansions
\[
y(\cdot,t)=\sum_{n=1}^{\infty}y_n(t)e_n,
\qquad
u(\cdot,t)=\sum_{n=1}^{\infty}u_n(t)e_n,
\]
with convergence in \(L^2(0,1)\).

Testing the weak state equation against each
eigenfunction \(e_n\) yields, for almost every
\(t\in(0,T)\),
\[
    y_n'(t)+\lambda_n y_n(t)=u_n(t),
\]
with initial condition
\[
    y_n(0)=y_{0,n},
\]
where
\[
    y_{0,n}=\int_0^1 y_0(x)e_n(x)\,dx
\]
denotes the \(n\)-th Fourier coefficient of the initial state.
Solving this linear ordinary differential equation gives
  \[
	y_n(t) = e^{-\lambda_n t}y_{0,n} + \int_0^t e^{-\lambda_n(t-s)}u_n(s)\,ds.
  \]
Imposing the terminal condition \(y(\cdot,T)=y_T\) leads to
\[
    y_{T,n}
    =e^{-\lambda_n T}y_{0,n}
    +\int_0^T e^{-\lambda_n(T-s)}u_n(s)\,ds,
\]
where
\[
    y_{T,n}  = \int_0^1 y_T(x)e_n(x)\,dx
\]
denotes the \(n\)-th Fourier coefficient of the terminal profile.
Define
  \[
	d_n = y_{T,n} - e^{-\lambda_n T}y_{0,n}.
  \]

Then
\[
d_n = \int_0^T e^{-\lambda_n(T-s)}u_n(s)\,ds
= \langle g_n,u_n\rangle_{L^2(0,T)},
\]
where
\[
g_n(s)=e^{-\lambda_n(T-s)}.
\]

We first establish the necessity of the
admissibility condition. Suppose that
\(u\in L^2(Q)\) satisfies \(y(\cdot,T)=y_T\).
By the Cauchy--Schwarz inequality,
\[
|d_n|^2
=|\langle g_n,u_n\rangle_{L^2(0,T)}|^2
\le
\|g_n\|_{L^2(0,T)}^2
\|u_n\|_{L^2(0,T)}^2.
\]
Since
\[
a_n=\|g_n\|_{L^2(0,T)}^2
=\frac{1-e^{-2\lambda_nT}}{2\lambda_n},
\]
we obtain
\[
\frac{|d_n|^2}{a_n}
\le \|u_n\|_{L^2(0,T)}^2.
\]
Summing over all modes and applying
Parseval's identity yields
\[
\sum_{n=1}^{\infty}\frac{|d_n|^2}{a_n}
\le
\sum_{n=1}^{\infty}
\|u_n\|_{L^2(0,T)}^2
=
\|u\|_{L^2(Q)}^2<\infty.
\]
Thus \(y_T\in\mathcal A_T(y_0)\), proving
that the admissibility condition is necessary.

We now establish sufficiency and construct
the unique minimum-energy control.

Using the orthonormality of \(\{e_n\}\),
the \(L^2(Q)\) norm of the control decomposes as
\[
\|u\|_{L^2(Q)}^2
=\sum_{n=1}^{\infty}
\|u_n\|_{L^2(0,T)}^2.
\]

Moreover, the terminal constraint also decouples across modes, yielding
independent conditions \(\langle u_n, g_n \rangle = d_n\) for each \(n\). Since both
the cost functional and the constraints are separable across modes, the
minimization problem reduces to solving each mode independently.

Thus, each modal component reduces to a minimum-energy problem of the form
  \[
    \min_{u_n \in L^2(0,T)}
    \frac12 \int_0^T |u_n(s)|^2\,ds
    \quad \text{subject to} \quad
    \langle u_n, g_n \rangle = d_n.
  \]

The set
\[
C_n=\{u_n:\langle u_n,g_n\rangle=d_n\}
\]
is an affine hyperplane with normal direction \(g_n\). Indeed, if
\(u_n,w\in C_n\), then
\[
\langle w-u_n,g_n\rangle=0,
\]
so every direction tangent to \(C_n\) is orthogonal to \(g_n\).

The closest point in \(C_n\) to the origin therefore lies along the normal
direction, and hence
\[
u_n^*\in \operatorname{span}\{g_n\}.
\]
Writing \(u_n^*(t) = c_n g_n(t)\) and imposing the constraint 
\(\langle u_n^*, g_n \rangle = d_n\), we obtain
  \[
	c_n \|g_n\|^2 = d_n, \quad \text{so} \quad c_n = \frac{d_n}{\|g_n\|^2}.
  \]
Therefore,
  \[
	u_n^*(t) = \frac{d_n}{\|g_n\|^2} g_n(t).
  \]

Recalling that
\[
a_n=\|g_n\|_{L^2(0,T)}^2
=\frac{1-e^{-2\lambda_nT}}{2\lambda_n},
\]
we obtain
  \[
	u_n^*(t) = \frac{2\lambda_n e^{-\lambda_n(T-t)}}{ 1-e^{-2\lambda_n T} } d_n.
  \]

The full optimal control is obtained by summing over all modes:
  \[
	u^*(x,t) = \sum_{n=1}^{\infty} 
	\frac{ 2\lambda_n e^{-\lambda_n(T-t)} }{1-e^{-2\lambda_n T} } d_n e_n(x).
  \]
Since
  \[
	y_T\in\mathcal A_T(y_0),
  \]
the admissibility condition implies
  \[
    \sum_{n=1}^{\infty}
    \frac{2\lambda_n}{1-e^{-2\lambda_n T}} |d_n|^2 <\infty.
  \]
Moreover, Parseval's identity gives
  \[
    \|u^*\|_{L^2(Q)}^2 = \sum_{n=1}^{\infty} \|u_n^*\|_{L^2(0,T)}^2.
  \]
Since 
  \[
    \|u_n^*\|_{L^2(0,T)}^2 = \frac{|d_n|^2}{\|g_n\|^2}  
    = \frac{2\lambda_n}{1-e^{-2\lambda_n T}}|d_n|^2
  \]
we have
  \[
    \|u^*\|_{L^2(Q)}^2 = \sum_{n=1}^{\infty}
    \frac{2\lambda_n}{1-e^{-2\lambda_n T}} |d_n|^2 <\infty.
  \]
Hence \(u^*\in L^2(Q)\).

Substituting the control into the modal representation verifies the terminal
condition as follows. The modal solution at time \(T\) is
  \[
	y_n(T) = e^{-\lambda_n T}y_{0,n} + \int_0^T e^{-\lambda_n(T-s)}u_n^*(s)\,ds.
  \]
We compute
  \[
    \int_0^T e^{-\lambda_n(T-s)}u_n^*(s)\,ds
    = \int_0^T g_n(s)\frac{d_n}{\|g_n\|^2}g_n(s)\,ds
    = \frac{d_n}{\|g_n\|^2}\int_0^T |g_n(s)|^2\,ds
    = d_n.
  \]
Therefore
  \[
	y_n(T) = e^{-\lambda_n T}y_{0,n}+d_n.
  \]
Since
  \[
	d_n=y_{T,n}-e^{-\lambda_n T}y_{0,n},
  \]
it follows that
  \[
	y_n(T)=y_{T,n}.
  \]
Thus each Fourier coefficient of \(y(\cdot,T)\) agrees with the corresponding
Fourier coefficient of \(y_T\), and hence
  \[
	y(\cdot,T)=y_T.
  \]
  
To prove uniqueness, let \(\tilde u\in L^2(Q)\)
be any control satisfying the same terminal
constraint. Then, for each \(n\),
\[
\langle \tilde u_n-u_n^*,g_n\rangle=0.
\]
Since \(u_n^*\) is proportional to \(g_n\),
we have
\[
\langle \tilde u_n-u_n^*,u_n^*\rangle=0.
\]
Therefore, by the Pythagorean identity,
\[
\|\tilde u_n\|_{L^2(0,T)}^2
=
\|u_n^*\|_{L^2(0,T)}^2
+
\|\tilde u_n-u_n^*\|_{L^2(0,T)}^2.
\]
Summing over all modes yields
\[
\|\tilde u\|_{L^2(Q)}^2
=
\|u^*\|_{L^2(Q)}^2
+
\|\tilde u-u^*\|_{L^2(Q)}^2.
\]
Consequently,
\[
\|\tilde u\|_{L^2(Q)}^2
\ge \|u^*\|_{L^2(Q)}^2,
\]
with equality if and only if
\(\tilde u=u^*\) in \(L^2(Q)\).
Thus \(u^*\) is the unique minimum-energy
control.

\end{proof}

%%%%%%%%%%%%%%%%%%%%%%%%%%%%%%%%%%%%%%%%%%%%%%%%%%%%%%%%%%%%

\section{Proof of Soft-to-Hard Convergence}

We prove the strong convergence of the penalized
optimal controls and their terminal states
established in Section~\ref{sec:soft-to-hard-convergence}.

\begin{proof}
Define
  \[
	a_n := \|g_n\|_{L^2(0,T)}^2 = \frac{1 - e^{-2\lambda_n T}}{2\lambda_n}.
  \]
From the explicit modal formula derived in Appendix~A,
  \[
	u_n^*(t) = \frac{d_n}{a_n} g_n(t).
  \]
For the penalized problem, the \(n\)-th modal functional is
  \[
	J_{\alpha,n}(u_n) 
	= \frac12\|u_n\|_{L^2(0,T)}^2 
	+ \frac{\alpha}{2} \left|\langle u_n,g_n\rangle-d_n\right|^2.
  \]
Its first-order optimality condition is
  \[
	u_{\alpha,n}^* 
	+ \alpha\bigl(\langle u_{\alpha,n}^*,g_n\rangle-d_n\bigr)g_n =0.
  \]
Hence \(u_{\alpha,n}^*\in \operatorname{span}\{g_n\}\), so we write
  \[
	u_{\alpha,n}^*(t)=c_{\alpha,n}g_n(t).
  \]
Substituting this into the optimality condition, we obtain
  \[
	c_{\alpha,n}+\alpha(c_{\alpha,n}a_n-d_n)=0.
  \]
Solving for \(c_{\alpha,n}\), we get
  \[
	c_{\alpha,n} = \frac{\alpha d_n}{1+\alpha a_n},
  \]
and consequently,
  \[
	u_{\alpha,n}^*(t) = \frac{\alpha d_n}{1+\alpha a_n} g_n(t).
  \]
We compute
  \[
    \begin{aligned}
	u_{\alpha,n}^*(t) - u_n^*(t) 
	& = \left( \frac{\alpha d_n}{1+\alpha a_n} -
	\frac{d_n}{a_n} \right) g_n(t) \\ 
	& = -\frac{d_n}{a_n(1+\alpha a_n)} g_n(t).
    \end{aligned}
  \]
Hence
  \[
    \|u_{\alpha,n}^*-u_n^*\|_{L^2(0,T)}^2
    = \frac{|d_n|^2}{a_n(1+\alpha a_n)^2}.
  \]

By spatial orthogonality and integration over time,
\[
\|u_\alpha^*-u^*\|_{L^2(Q)}^2
=
\sum_{n=1}^{\infty}
\frac{|d_n|^2}{a_n(1+\alpha a_n)^2}.
\]
For each fixed \(n\), the summand tends to zero as \(\alpha\to\infty\).
Moreover,
  \[
	0 \le \frac{|d_n|^2}{a_n(1+\alpha a_n)^2} \le \frac{|d_n|^2}{a_n},
  \]
and the right-hand side is summable by the admissibility condition. 
Therefore, by the dominated convergence theorem for series,
  \[
    \|u_\alpha^*-u^*\|_{L^2(Q)}^2 \to 0.
  \]

Using the modal representation and the formula for \(u_{\alpha,n}^*\), 
we have
\[
\begin{aligned}
    y_{\alpha,n}(T) 
    &= e^{-\lambda_n T}y_{0,n} +\int_0^T g_n(s)u_{\alpha,n}^*(s)\,ds \\
    &= e^{-\lambda_n T}y_{0,n} 
        +\frac{\alpha d_n}{1+\alpha a_n} \int_0^T |g_n(s)|^2\,ds \\
    &= e^{-\lambda_n T}y_{0,n} + \frac{\alpha a_n}{1+\alpha a_n}d_n.
\end{aligned}
\]
Since \(d_n=y_{T,n}-e^{-\lambda_n T}y_{0,n}\), it follows that
\[
    y_{\alpha,n}(T)-y_{T,n} = -\frac{1}{1+\alpha a_n}d_n.
\]
By Parseval's identity,
  \[
    \|y_\alpha(\cdot,T)-y_T\|_{L^2(0,1)}^2
    = \sum_{n=1}^{\infty} \frac{|d_n|^2}{(1+\alpha a_n)^2}.
  \]

For each fixed \(n\), the summand tends to zero
as \(\alpha\to\infty\). Moreover,
\[
0\le
\frac{|d_n|^2}{(1+\alpha a_n)^2}
\le |d_n|^2.
\]
Since \(d\in L^2(0,1)\), Parseval's identity gives
\[
\sum_{n=1}^{\infty}|d_n|^2<\infty.
\]
The dominated convergence theorem for series
therefore implies
\[
\|y_\alpha(\cdot,T)-y_T\|_{L^2(0,1)}^2
\longrightarrow 0.
\]
Hence
\[
y_\alpha(\cdot,T)\longrightarrow y_T
\quad\text{in }L^2(0,1).
\]

\end{proof}

%%%%%%%%%%%%%%%%%%%%%% %%%%%%%%%%%%%%%%%%%%%%%%%%%%%%%%%%%%%%%%%
	
\section{Proof of the Rates of Convergence}

We establish the quantitative convergence
estimates stated in Section~\ref{sec:rates-of-convergence}.
\begin{proof}

Using the modal identities derived in Appendix~B,
we have
\[
    y_{\alpha,n}(T)-y_{T,n} = -\frac{d_n}{1+\alpha a_n},
\]
and
\[
    \|u_{\alpha,n}^*-u_n^*\|_{L^2(0,T)}^2 = \frac{|d_n|^2}{a_n(1+\alpha a_n)^2}.
\]

We first estimate the terminal-state error. By Parseval's identity,
\[
    \|y_\alpha(\cdot,T)-y_T\|_{L^2(0,1)}^2 
    = \sum_{n=1}^{\infty} \frac{|d_n|^2}{(1+\alpha a_n)^2}.
\]
Let \(0\le\theta\le\frac12\). Since
\[
    1+\alpha a_n \ge \alpha a_n,
\]
raising both sides to the power \(1+2\theta\) gives
\[
    (1+\alpha a_n)^{1+2\theta} \ge (\alpha a_n)^{1+2\theta}.
\]
Because \(1+2\theta\le2\),
\[
    (1+\alpha a_n)^2 \ge (1+\alpha a_n)^{1+2\theta}.
\]
Therefore,
\[
    \frac{1}{(1+\alpha a_n)^2} \le \frac{1}{\alpha^{1+2\theta}a_n^{1+2\theta}}.
\]
Hence
\[
    \|y_\alpha(\cdot,T)-y_T\|_{L^2(0,1)}^2 \le \frac{1}{\alpha^{1+2\theta}}
    \sum_{n=1}^{\infty} \frac{|d_n|^2}{a_n^{1+2\theta}}.
\]
It follows that
\[
    \|y_\alpha(\cdot,T)-y_T\|_{L^2(0,1)} 
    \le \frac{C_\theta}{\alpha^{\frac12+\theta}},
\]
where
\[
    C_\theta = \left( \sum_{n=1}^{\infty} 
    \frac{|d_n|^2}{a_n^{1+2\theta}} \right)^{1/2}.
\]

We now estimate the control error. By spatial orthogonality and
integration over time,
\[
    \|u_\alpha^*-u^*\|_{L^2(Q)}^2 = \sum_{n=1}^{\infty}
    \frac{|d_n|^2}{a_n(1+\alpha a_n)^2}.
\]
Let \(0<\theta\le1\). Since
\[
    1+\alpha a_n \ge \alpha a_n,
\]
raising both sides to the power \(2\theta\) gives
\[
    (1+\alpha a_n)^{2\theta} \ge (\alpha a_n)^{2\theta}.
\]
Because \(2\theta\le 2\),
\[
    (1+\alpha a_n)^2 \ge (1+\alpha a_n)^{2\theta}.
\]
Therefore,
\[
    \frac{1}{a_n(1+\alpha a_n)^2} \le \frac{1}{\alpha^{2\theta}a_n^{1+2\theta}}.
\]
Hence
\[
    \|u_\alpha^*-u^*\|_{L^2(Q)}^2 \le \frac{1}{\alpha^{2\theta}}
    \sum_{n=1}^{\infty} \frac{|d_n|^2}{a_n^{1+2\theta}}.
\]
It follows that
\[
    \|u_\alpha^*-u^*\|_{L^2(Q)} \le \frac{D_\theta}{\alpha^\theta},
\]
where
\[
    D_\theta 
    = \left( \sum_{n=1}^{\infty} \frac{|d_n|^2}{a_n^{1+2\theta}} \right)^{1/2}.
\]

This proves both estimates.

\end{proof}

\begin{remark}
In the terminal-state estimate, the case \(\theta=0\) corresponds precisely to
the admissibility condition
\[
\sum_{n=1}^{\infty}\frac{|d_n|^2}{a_n}<\infty,
\]
and yields the convergence rate
\[
\|y_\alpha(\cdot,T)-y_T\|_{L^2(0,1)}
=
O(\alpha^{-1/2}).
\]
Taking \(\theta=\frac12\) gives the stronger estimate
\[
\|y_\alpha(\cdot,T)-y_T\|_{L^2(0,1)}
=
O(\alpha^{-1})
\]
under the stronger summability condition
\[
\sum_{n=1}^{\infty}\frac{|d_n|^2}{a_n^2}<\infty.
\]

Similarly, taking \(\theta=1\) in the control estimate yields
\[
\|u_\alpha^*-u^*\|_{L^2(Q)}
=
O(\alpha^{-1})
\]
provided that
\[
\sum_{n=1}^{\infty}\frac{|d_n|^2}{a_n^3}<\infty.
\]
\end{remark}

\bibliographystyle{plain}
\bibliography{soft_to_hard_refs}

\end{document}